\documentclass[12pt]{amsart}
\usepackage[usenames]{color}
\usepackage{subfigure}
\usepackage{amssymb, amsmath, amsfonts,amsthm,mathrsfs}

\usepackage{graphicx}
\usepackage[colorlinks=true]{hyperref}
  \hypersetup{urlcolor=blue, citecolor=red}

\advance\textheight 0.2cm

\newtheorem{Theorem}{Theorem}
\newtheorem{theorem}{Theorem}[section]
\newtheorem{lemma}[theorem]{Lemma}

\newtheorem{corollary}[theorem]{Corollary}

\theoremstyle{definition}

\numberwithin{equation}{section}

\def\F{{\mathscr F}}
\def\C{C}
\def\Chi{\mathcal X}
\def\HH{\mathcal H}
\def\S{S}

\def\RR{{\mathbb R}}

\def\Ln{\lambda_n}
\def\esup { {\rm ess} \sup}
\def\supp { {\rm supp}\,}

\def\BV{\mathop{BV_{\rm loc}}}
\def\W11{\mathop{W^{1,1}_{\rm loc}}}
\def\Per{{\rm Per}\,}

\title[Extremals of the 
P\'olya-Szeg\H{o} inequality]
{On the extremals of the \\P\'olya-Szeg\H{o} inequality}
\author{Almut  Burchard and Adele Ferone}
\address{Department of Mathematics --- University of Toronto, 
40 St. George Street, Toronto, Canada M5S 2E4}
\email{almut@math.toronto.edu}
\address{Dipartimento di Matematica --- Seconda Universit\`a di Napoli,
Viale Lincoln 5, 81100 Caserta, Italy}
\email{adele.ferone@unina2.it}

\begin{document}

\begin{abstract} The distance of an extremal 
of the P\'olya-Szeg\H{o} inequality 
from a translate of its symmetric decreasing 
rearrangement is controlled by the 
measure of the set of critical points. 
\end{abstract}

\maketitle

\thispagestyle{empty}

%%%%%%%%%%%%%%%%%%%%%%%%%%%%%%%%%%%%%%%%%%
\section{Introduction}
\label{sec:intro}

\noindent Let $u$ be a nonnegative function on $\RR^n$
that vanishes at infinity.
Many geometric inequalities
relate $u$ with its symmetric decreasing rearrangement, 
$u^\star$. 
The {\em P\'olya-Szeg\H{o} inequality} states that
\begin{equation}\label{PS}
\| \nabla u^\star \|_p\le \| \nabla u\|_p
\end{equation}
for every $1\le p\le \infty$ such that the distributional gradient
$|\nabla u|$ lies in $L^p$;
in particular, $|\nabla u^\star|$ again lies in $L^p$~\cite{PS}.
For $p=1$, this reduces to the
isoperimetric inequality, and for $p=\infty$, it follows from 
the fact that symmetric decreasing rearrangement
improves the modulus of continuity.

{
Inequality~\eqref{PS} has been extended in various directions.
It holds for general convex Dirichlet-type functionals~\cite{BZ, Ta97},
on the larger space of functions that are locally of bounded
variation~\cite{CiFu1}, and with other symmetrizations in place
of the symmetric decreasing rearrangement~\cite{AFLT, Ka, B, Br1}. 
The functionals that satisfy general P\'olya-Szeg\H{o} inequalities
have been fully characterized; they are
known to include all rearrangement-invariant norms~\cite[Theorem 1.2]{CiFe}.}

In this paper, we study functions that
produce equality in \eqref{PS} for some $p$ with $1<p<\infty$. 
Such a function will be called an {\em extremal} of the inequality.

Extremals of \eqref{PS}
were first analyzed by Brothers and Ziemer in 1988~\cite{BZ}.
Clearly, every translate of a symmetric decreasing function 
is an extremal. In the converse direction,
the level sets of extremals must
be balls, but they need not be concentric.  For example, a 
function whose graph consists of a small cone stacked
on the frustrum of a large cone is an extremal, regardless of 
the precise position of the smaller cone on the plateau.
Brothers and Ziemer discovered that a
similar phenomenon can occur even for functions without plateaus.
Under the assumption that the distribution function
of $u$ is absolutely continuous, they proved that the only extremals 
are translates of $u^\star$.  Otherwise, there exist 
extremals that are equimeasurable to,
but not translates of~$u^\star$. 

The condition that the distribution function
be absolutely continuous is equivalent to requiring
that the set of non-trivial critical points of $u^\star$
has measure zero.  What can be said about extremals
where this set has positive measure?
In 2006, Cianchi and Fusco proved that every
extremal of \eqref{PS} whose support has finite measure 
satisfies
\begin{equation}
\label{eq:CF}
\|u-u^\star\circ\tau\|_1 \le 
L_n ||\nabla u||_p \cdot
\lambda_n(\supp u)^{\frac{1}{p'}+\frac{2n-1}{2n^2}}
\,\lambda_n(\C)^{\frac{1}{2n^2}}
\end{equation}
for a suitable translation $\tau$~\cite[Theorem 1.1]{CiFu2}.
Here, $L_n$ is a constant that depends 
on the dimension,  
$\lambda_n$ is the $n$-dimensional Lebesgue measure, 
$p'=p/(p\!-\!1)$ is the H\"older dual exponent of $p$, 
and $\C$ is the set of critical points defined
in Theorem~\ref{main} below.
Our goal is to simplify the 
analysis and construct explicit bounds for extremals whose 
support need not have finite measure.

The results of Brothers-Ziemer and Cianchi-Fusco apply 
to certain convex Dirichlet-type functionals that will be 
described below, and to the more general functionals 
treated in~\cite{CiFe}, which need not be in integral form.  
They remain valid --- after adjusting the constants ---
for the convex rearrangement, which replaces
the level sets of $u$ by suitably scaled copies
of a {centrally symmetric} convex body $B\subset\RR^n$.
For the sake of simplicity, we 
focus on the classical case of the $L^p$-norm of the gradient 
with $1<p<\infty$, leaving the discussion of more
general functionals for the last section of the paper.

\begin{Theorem}\label{main} Let $u$ be a nonnegative function on $\RR^n$ 
that vanishes at infinity and whose distributional gradient
lies in $L^p$ for some $1<p<\infty$,
and let $u^\star$ be its symmetric decreasing rearrangement.
If
$$
\| \nabla u \|_p =\| \nabla u^\star \|_p\,,
$$
then there exists a translation $\tau$ such that
\begin{equation}
\label{eq:main}
\| u-u^\star\circ \tau\|_q\le K_n^{1/q} \, ||u||_q^{1/n'}\,
||u\Chi_\C||_q^{1/n}\,
\end{equation}
for every $q\ge 1$ with $u\in L^q$.  Here, $K_n=2\omega_{n-1}/\omega_n$, and
\begin{equation*}
\C=  \bigl\{x\in \RR^n : 
0< u(x)<\esup u\>\> \mbox{and}\>\> |\nabla u(x)| = 0\bigr\}\,.
\end{equation*}
\end{Theorem}

\smallskip
The translation $\tau$ is chosen to align the graphs of
$u^\star\circ \tau$ {and} $u$ at the top.
The value of the constant is given by
$K_n=\bigl(\int_0^{\pi/2} \cos^n\theta\, d\theta\bigr)^{-1}
\sim (2n/\pi)^{1/2}$.
The set $\C$ consists of the non-trivial critical points of $u$,
except for the possible plateau at height $\esup u$. 
For the conclusion of the theorem, we can equivalently 
replace the function $u$ by $u^\star$ on the right
hand side of \eqref{eq:main} and in the definition of $\C$. 
Indeed, if $u\in W^{1,1}_{loc}$ and $\S$ is the set 
where the singular part of the 
distribution function is concentrated, then 
$\C\supset u^{-1}(\S)$ in general, and equality holds 
if $u$ is an extremal.

If $\C$ has finite measure,
there is a simpler estimate in terms of
its volume radius.

\begin{Theorem}\label{finite} 
Under the assumptions of Theorem~\ref{main}, if 
$\Ln(\C)<\infty$ then there exists a translation $\tau$ such that
\begin{equation}\label{eq:finite}
|| u-u^\star\circ\tau ||_p\le 
||\nabla u||_p\cdot \left(\frac{\Ln(\C)}{\omega_n}\right)^{\frac1n}.
\end{equation}
\end{Theorem}

For $1<p<n$, a natural choice for $q$ in Theorem~\ref{main}
is the Sobolev exponent
$p^*=np/(n\!-\!p)$, for which the
right hand side of \eqref{eq:main} is bounded by the
Sobolev inequality. Interpolating 
with \eqref{eq:finite} yields $L^q$-bounds
for $p<q<p^*$, provided that $\C$ has finite measure. 
For $p>n$, there is a corresponding bound in $L^\infty$.

\begin{Theorem}\label{Morrey} 
Under the assumptions of Theorem \ref{main}, if $p>n$
and $\lambda_n(\C)<\infty$ then 
there exists a translation $\tau$ such that
$$
||u-u^\star\circ \tau||_\infty \le 
M_{n,p} \,
\| \nabla u\|_p\cdot 
\left(\frac{\Ln(\C)}{\omega_n}\right)^{\frac1n-\frac1p}\,,
$$ 
where $M_{n,p}$ is the Morrey constant.
\end{Theorem}

\smallskip 
We briefly describe the relation
with the literature.
Theorem~\ref{main} contains the result of Brothers and Ziemer,
because the right hand side of \eqref{eq:main}
vanishes when the distribution function of $u$ is absolutely continuous.
Similarly, Theorem~\ref{finite} contains the bound of Cianchi and 
Fusco. To see this, apply
H\"older's inequality on the left hand side of 
\eqref{eq:finite} and use that
$\C\subset \supp u$ on the right to obtain \eqref{eq:CF} with
$L_n=2^{1/p'}\omega_n^{-1/n}$.
We will show below that the proof of Theorem~\ref{main}
also implies~\eqref{eq:CF}.
However, Theorems~\ref{finite} and~\ref{Morrey} do
not seem to follow directly from Theorem~\ref{main}.

{
\medskip\noindent{\em Acknowledgments.}
Research for this paper was supported in part
by an NSERC Discovery Grant and a GNAMPA Project.
%
% We thank  the Government of Canada
%for partial support through NSERC Discovery Grant \# 311685-10 and  through  the G.N.A.M.P.A. Research  Project 2013 {\it Problemi sovradeterminati, problemi inversi e disuguaglianze funzionali: stabilit\`a}.
}

\section{Outline of the proof}  
Brothers and Ziemer characterized extremals as follows.
If $u$ satisfies the assumptions of Theorem~\ref{main}, 
then its level sets are balls,
\begin{equation}\label{eq:balls}
\{u>t\}= \xi_t + \{u^\star>t\}
\end{equation} 
(up to sets of Lebesgue measure zero). Furthermore,
the gradient is equidistributed on level sets,
\begin{equation}\label{eq:grad}
|\nabla u(x)|_{\rfloor \partial \{u>t\}} 
= |\nabla u^\star |_{\rfloor \partial \{u^\star>t\}}
\end{equation}
for $\HH^{n-1}$-almost every $x\in\partial \{u>t\}$
and almost every $t\in (0,\esup u)$.
This equidistribution property is a consequence of the strict 
convexity of the function $t\to t^p$. All later work 
on the problem relies on this characterization.

A more delicate issue is to prove that the level 
sets are {\em concentric} balls
if the distribution function of $u$ is absolutely 
continuous. Brothers and Ziemer, starting from 
\eqref{eq:balls}, express $u^\star$ in terms of $u$ as 
$u^\star=u\circ T$ and study the regularity of the transformation 
$T$ under the assumption of the {\em continuity} of the 
distribution function. The crucial point is the 
evaluation of $\nabla u^\star$, which
requires a non-standard chain rule 
because $u$ is just a Sobolev function.
The {\em absolute
continuity} of the distribution function is needed to 
deduce that $T$ is a translation of the identity map 
from \eqref{eq:grad} and the fact that
$\nabla u^\star (x)=(DT(x))^t \,\nabla u (T(x))$.

In the last ten years, several new proofs of these
results have appeared.
In~\cite{FV1}, the authors reverse the approach of Brothers 
and Ziemer and express
$u$ in terms of $u^\star$ as $u(x)= u^\star (T(x))$.  
That leads to an easier case
of the chain rule, because $u^\star$ is essentially
a function of a single variable.
Finally, the conclusion is obtained by a gradient-flow argument. 
Since this last part of the proof relies
on the uniform convexity and smoothness properties
of the Euclidean norm, the authors later
developed yet another geometric argument to treat
rearrangements with respect to arbitrary 
%non-Euclidean 
norms in $\RR^n$~\cite{FV2}. 
Their method was subsequently used by
Cianchi and Fusco in~\cite{CiFu2}.

The argument in ~\cite{FV2} proceeds as follows.
Let $\xi_t$ denote the center of the ball $\{u>t\}$,
and let $R(x)$ be the function that assigns
to each point  $x\in\RR^n$ the radius
of the ball $\{u>u(x)\}$. The key
observation is that for all $s,t\in (0,\esup u)$
there exists a pair of points
$x\in\partial\{u>s\}$ and $y\in\partial\{u>t\}$
such that
\begin{equation} \label{eq:xi-Phi}
|\xi_s-\xi_t| = 
|R(x)-R(y)| -|x-y|\,,
\end{equation}
see Fig.~\ref{fig:key}. 
If the distribution function of $u$ is absolutely 
continuous, then 
$R$ is Lipschitz continuous and
$|\nabla R|\equiv 1$.
It follows that $\xi_t$ is constant, proving that
$u$ is a translate of $u^\star$.

\begin{figure} [htbp]
\centering
\includegraphics{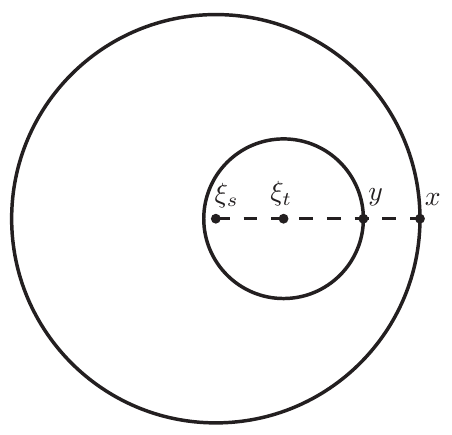}
\caption{\small Two circles ordered by inclusion.
The difference between the radii can be expressed 
as the sum of  their distance $|x-y|$ and the distance
of the centers $|\xi_s-\xi_t|$, see Eq.~\eqref{eq:xi-Phi}.}
\label{fig:key}
\end{figure}

If the distribution function of $u$ is not absolutely
continuous, then $R$ is of bounded variation. 
In~\cite{CiFu2}, the 
distribution function of $u$ is approximated by
an absolutely continuous function. Instead, we 
approximate $u$ by functions whose distribution functions 
have jumps, but no singular continuous part. In Section~\ref{sec:balls}, 
we analyze the variation of $R$ 
for such functions and derive the bound
\begin{equation}
\label{eq:var-xi}
|\xi_s-\xi_t|\le \left( \frac{1}{\omega_n}\mu^s ((s,t])\right)^{\frac1n}\,,
\end{equation}
where $\mu^s$ is the singular part of the measure 
associated with the distribution function of $u$.
It follows that the total variation of $\xi$ is
bounded by the volume radius of $\C$,
\begin{equation}
\label{eq:TV-xi}
||D\xi|| \le \left(\frac{\lambda_n(\C)}{\omega_n}\right)^{\frac1n}\,.
\end{equation}
In Section~\ref{sec:main}, we show that this
implies the main results. 
The final Section~\ref{sec:Dirichlet}
is dedicated to convex Dirichlet functionals.

Before turning to the technical part,
observe that the characterization
of extremals given by Brothers and Ziemer
does not depend on the value of $p$.
If $u$ is an extremal for {\em some} $1<p< \infty$,
then \eqref{eq:balls}-\eqref{eq:grad} imply by
the coarea formula that $u$ produces equality
in \eqref{PS} for {\em every} $1\le p\le \infty$.
In fact, much more is true. According 
to~\cite[Theorem 1.7]{CiFe}, there is a wide class of 
functionals satisfying a suitable strict monotonicity condition
that have the same family of extremals, which all satisfy 
the conclusions of Theorem~\ref{main}-\ref{Morrey} 
(as well as those of Corollaries~\ref{cor:main}-\ref{cor:finite} 
with $V=\emptyset$).

%%%%%%%%%%%%%%%%%%%%%%%%%%%%%%%%%%%%%%%%  
\section{Notation and preliminary results}
\label{sec:def}

We work on $\RR^n$, equipped with the 
standard Euclidean norm $|\cdot|$ and 
Lebesgue measure $\lambda_n$.
Let $u$ be a nonnegative measurable function on~$\RR^n$.
We always assume that $u$  {\it vanishes at infinity}, in the
sense that its level sets $\{x\in\RR^n: u(x)>t\}$ have finite
measure for all $t>0$. Its {\it distribution function} 
is given by
\begin{equation*}
F(t)=\Ln  (\{x\in\RR^n : |u(x)|>t\})\,, \qquad t>0\,.
\end{equation*}
The {\it symmetric decreasing rearrangement} of $u$ is defined by
\begin{equation}
\label{eq:def-ustar}
u^\star (x)=\sup\{t>0: F(t)>\omega_n|x|^n\}\qquad x\in\RR^n\,;
\end{equation}
it is the unique radially decreasing
function that is equimeasurable to $u$ and lower semicontinuous.

The following construction removes
a collection of horizontal slices from the graphs of $u$ and $u^\star$ 
(see Fig.~\ref{fig:approx}).
Given a finite or countable union of intervals
$I\subset \RR_+$,
set 
\begin{equation}
\label{eq:slice}
f(t)=\lambda_1([0,t]\setminus I)\,.
\end{equation}
Then $f\circ u$ vanishes at infinity, and
$(f\circ u)^\star = f \circ u^\star$.  If $u\in \W11$,
then $f\circ u\in W^{1,1}_{\rm loc}$, and
$$
\nabla (f\circ u)(x) = \Chi_{\{u(x)\not \in I\}} \nabla u(x)
$$
almost everywhere on $\RR^n$ (see~\cite[Corollary 6.18]{LL}).

\begin{figure}[htbp]
\centering%
\includegraphics{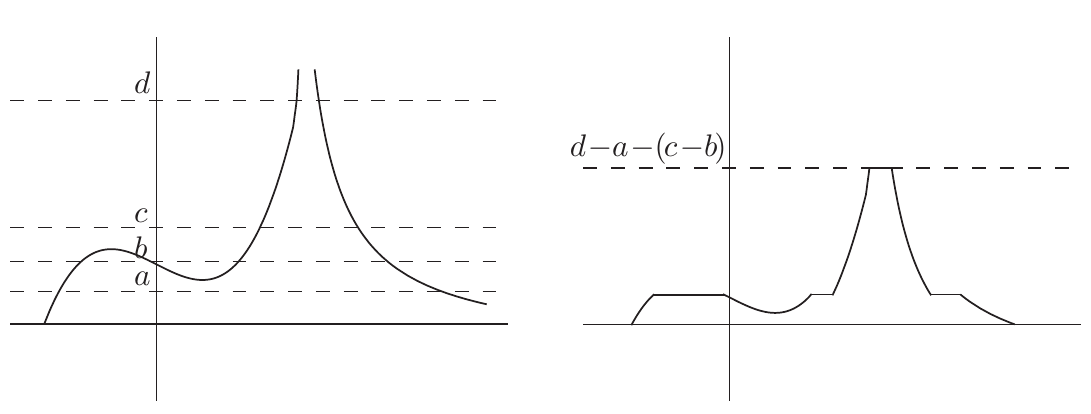}
\caption{\small Removing a horizontal slice from the
graph of $u$, see Eq.~\eqref{eq:slice}.
The right panel shows
$f\circ u$ with $I=[0,a)\cup(b,c)\cup (d,\infty)$.
}
\label{fig:approx}
\end{figure}

Many useful quantities can be expressed in terms of distribution
functions. For any absolutely continuous
function $\Psi$ on $\RR_+$ with $\Psi(0)=0$, 
there is the {\em layer-cake representation}
$$
\int \Psi(u)\, dx = \int_0^\infty F(t)\,\Psi'(t)\, dt\,.
$$
We note for later use that
\begin{equation}
\label{eq:L1-dist}
||u-v||_1= \int_0^\infty \lambda_n(\{u>t\} \bigtriangleup
\{v>t\})\, dt\,
\end{equation}
for any pair of nonnegative integrable functions 
(here {$\bigtriangleup$}
stands for the symmetric difference of sets).
The following two lemmas provide similar formulas
for other convex functions of $|u-v|$.

\begin{lemma}
\label{lem:Psi-1}
Let $\Psi$ be a convex function on $\RR_+$
with $\Psi(0)=\Psi'(0)=0$, let $\nu$ be the measure
that represents its second distributional derivative,
and let $u,v$ be nonnegative measurable functions. Then
\begin{align}
\notag \int \Psi(|u-v|)\, dx
& =\int_0^\infty\!\!\int_0^t
\bigl[
\lambda_n(\{u>t\}\setminus \{v>t\!-\!s\})\\
\label{eq:Psi-dist-1}
& \qquad\qquad +\lambda_n(\{v>t\}\setminus\{u>t-s\})\bigr]
\, d\nu(s) dt\,.
\end{align}
\end{lemma}

\begin{proof} We use that
\begin{align*}
\Psi(b-a) = \int_a^b \!\!\int_0^{t-a}\, d\nu(s)dt 
= (\nu\!\times\!\lambda_1)
\bigl(\{(s,t): a+s<t<b\}\bigr)\,
\end{align*}
for $b>a>0$,  and split the integral according to the sign of $u-v$,
$$
\int \Psi(|u-v|) \, dx = \int_{\{v>u\}} \Psi(v-u)\, dx
+\int_{\{u>v\}} \Psi(u-v)\, dx \,.
$$
For the first integral on the right, Fubini's theorem
gives
\begin{align*}
\int_{\{v>u\}} \Psi(v-u)\, dx
&= (\nu\!\times\! \lambda_1\!\times\!\lambda_n)
\bigl(\{
(s,t,x): u(x)+s <t<v(x)\}\bigr)\\
&=\int_0^\infty\!\!\int_0^t \lambda_n\bigl(\{v>t\}\setminus \{u>t-s\}\bigr)\,
d\nu(s)dt\,.
\end{align*}
Treating the second integral in the same way, we 
arrive at the claimed identity.
\end{proof}

\begin{lemma}
\label{lem:Psi-2}
Let $\Psi$ be a convex function on $\RR_+$
with $\Psi(0)=0$,
and let $u,v$ be nonnegative measurable functions. Then
$$
\int \Psi(|u-v|)\, dx
\le \int_0^\infty \lambda_n(\{u>t\} \bigtriangleup
\{v>t\})\, \Psi'(t)\, dt\,.
$$
\end{lemma}

\begin{proof} Since the claim holds for {\em linear} 
functions $\Psi$ by \eqref{eq:L1-dist}, we may assume, by
replacing $\Psi(t)$ with $\Psi(t)-t\Psi'(0_+)$ that
$\Psi'(0)=0$. Then we can apply Lemma~\ref{lem:Psi-1}.
The right hand side of \eqref{eq:Psi-dist-1} increases if
we set $s=0$ in the integrand. We then evaluate
the inner integral, using that $\nu(0,t)=\Psi'(t)$ for almost every $t$.
\end{proof}

\smallskip 

Since the distribution function of $u$ is monotonically decreasing, it
defines a Borel measure on $\RR_+$ by
\begin{equation}\label{mon}
\mu ((a,b])=F(a)-F(b) = 
\lambda_n\bigl(u^{-1}(a,b]\bigr)\,.
\end{equation}
If $u$ is essentially bounded, we restrict this 
measure to the interval $(0,\esup u)$, neglecting
plateaus at $t=0$ and $t=\esup u$.
Consider the Lebesgue-Radon-Nikodym decomposition
$\mu=\mu^{ac}+\mu^s$, where $\mu^{ac}\ll \lambda_1$
and $\mu^s\perp \lambda_1$. This gives rise to a decomposition
of the distribution function
$F=F^{ac}+F^s$, where $F^s(t)=\mu^s(t,\esup u)$
and $F^{ac}$ is absolutely continuous.
By the Fundamental Theorem of Calculus, the density of
$\mu^{ac}$ is given by 
the classical derivative
\begin{equation}\label{eq:F-prime}
-F'(t)=\frac{\Per\left (\{u^\star >t\}\right )}
{|\nabla u^\star |_{\rfloor \partial \{u^\star>t\}}}\,,
\end{equation}
and the derivative of $F^{s}$ vanishes almost everywhere.
The singular part of the measure is given by
$$
\mu^s((a,b]) = F^s(a)-F^s(b)=\lambda_n\bigl(\{x\in (u^\star)^{-1}\left ((a,b]\right ):
\nabla u^\star(x)=0\}\bigr)\,,
$$
it is supported on the set of singular values
$$
\S = \bigl\{t\in (0,\esup u): 
F\ \mbox{is not differentiable at}\ t\bigr\}\,.
$$
Since $u^\star$ is a monotone function of the single radial variable,
the set $\S$ has measure zero.
Clearly $\mu$ is absolutely continuous if and only if $(u^\star)^{-1}(\S )$ 
has measure zero. 
Since the gradient of $u$ vanishes almost everywhere on 
$u^{-1}(\S)$~\cite[Theorem 6.19]{LL}, 
its follows that $\C\supset u^{-1}(\S)$ up
to a set of measure zero, and
\begin{equation}\label{eq:coarea}
\mu^s \left ( (a,b]\right ) \le 
\lambda_n\bigl(\{x\in u^{-1}\left ((a,b]\right ):
\nabla u (x)=0\}\bigr)\,.
\end{equation}
{Note} that strict inequality can occur: when equality holds 
in \eqref{eq:coarea} $u$ is said to be {\em{coarea regular}} 
\cite{AlLi}. The next result shows that every extremal 
is coarea regular (see also \cite{B, CiFu1}), and allows
us to interpret $F^s(t)$ as the distribution
function of the restriction of $u$ to $\C$.

\begin{lemma} \label{lem:C-S}
Under the assumptions of Theorem~\ref{main},
$\C = u^{-1}(\S)$ up to a set of measure zero. 
Furthermore,
$$
F^s(t) =
\lambda_n(\{x\in \C: u(x)>t\}) \,.
$$
\end{lemma}

\begin{proof}
Since $\C\supset u^{-1}(\S)$, we have only to prove the reverse inclusion.
The coarea formula and the characterization
of extremals in \eqref{eq:balls} and \eqref{eq:grad}
show that
\begin{align*}
\lambda_n(\{x\not\in\C: u(x)>t\})
&= \int_t^\infty \int_{\partial\{u>t\}}
|\nabla u|^{-1} d\HH^{n-1} dt\\
&=\int_t^\infty |\nabla u^\star|^{-1} \Per(\{u^\star>t\})\, dt\\
&= \lambda_n(\{x: \nabla u^\star(x)\ne 0, u^\star(x)>t\})\,.
\end{align*}
Since $u$ and $u^\star$ are equimeasurable, it follows that
\begin{align*}
\lambda_n(\{x\in \C: u(x)>t\}) 
&= \lambda_n(\{x: \nabla u^\star(x)=0, u^\star(x)>t\})\\
&= F^s(t)\,.
\end{align*}
\vskip -1\baselineskip
\end{proof}

\smallskip

In general, $\mu^s$ can be further decomposed
into a sum of (at most) countably many point masses 
that correspond to plateaus of $u$,
and a singular continuous component.
However, one can always
approximate $u$ by functions
whose distribution function has no singular continuous part.

\begin{lemma} (Approximation)
\label{lem:approx}
Let $u\in \W11$ be a nonnegative function
that vanishes at infinity.
There exists an increasing sequence of functions
$u_m\in\W11$ and a decreasing sequence
of sets $\C_m\subset \RR^n$ with
$$
\lim u_m=u\,, \qquad \bigcap \C_m=u^{-1}(\S)\,,
$$
such that $u_m$ is bounded and supported on a set of
finite measure, each level set $\{u_m>t\}$ is also a 
level set of $u$, the distribution function
of $u_m$ has no singular continuous part, and
$$
\nabla u_m = \Chi_{\C_m}\nabla u\,,\qquad (m\ge 1)\,.
$$
\end{lemma}

\begin{proof}
Since $\mu^s$ is a regular Borel measure,
there exists a decreasing sequence of open sets $\S_m$
containing $\S$ such that
$$
\lim \mu(\S_m\cap (t,\esup u)) =\mu^s(t,\esup u)
$$
for all $t>0$. If $u$ is unbounded or the support of $u$ does
not have finite measure, we ask
that $\S_m\supset [0,1/m)\cup (m,\infty)$.
Set $f_m(t) = \lambda_1([0,t]\setminus\S_m)$, let $u_m=f_m\circ u$,
and let $\C_m=u^{-1}(\S_m)$.
Then $u_m$ is bounded and supported on a set of finite measure.
Moreover, since $\S_m$ is open, it is the union
of (at most) countably many disjoint intervals.
Therefore 
$$
\nabla u_m(x) = \Chi_{u(x)\not\in \S_m}\nabla u(x)\,
$$
for almost every $x$. 
By construction, $u_m$ increases
monotonically to $u$, and $|\nabla u_m|$ increases
to $|\nabla u|$.  For each connected component $(a,b)$ of $\S_m$, 
the distribution function 
of $u_m$ has a jump of size $F(a)-F(b)$, corresponding
to a plateau of $u_m^\star$.  Since $\S_m\supset \S$,
the distribution function of $u_m$ has no singular 
continuous component.
\end{proof}

\smallskip 

We will also consider functions on $\RR^n$
that do not lie in $\W11$ but in the larger space 
$\BV$. A function
$u$ is locally of bounded variation, if its
distributional derivative is represented by
a vector-valued Radon measure, $[Du]$. 
We denote by $|Du|$ the corresponding variation measure,
and by $||Du||=|Du|(\RR^n)$ its total variation.
The variation measure has Lebesgue-Radon-Nikodym representation
$[Du] = [D^{ac}u] + [D^su]$, where the
absolutely continuous component has density $\nabla u$, and $[D^su]$
is supported on a set of Lebesgue measure zero.
We will always use the {\em precise representative} 
of $u$ that agrees
with its Lebesgue density limit at every point where 
it exists.  
For more information about $\BV$, we refer the reader to
to~\cite{EG,AFP}.

If $u$ is a nonnegative function in $\BV$ that vanishes at
infinity, then $u^\star\in \BV$. Since $u^\star$ is a monotone
function of the radius,
its singular continuous component
is supported on $(u^\star)^{-1}(\{t\in (0,\esup u):F'(t)=0\})$.

%%%%%%%%%%%%%%%%%%%%%%%%%%%%%%%%%%%%%%%%%
\section{Properties of extremals}
\label{sec:balls}

Throughout this section,
we assume that $u$ is an extremal for the P\'olya-Szeg\H{o} 
inequality~\eqref{PS}.  The goal is to prove
the bounds on the variation of $\xi$ in \eqref{eq:var-xi}
and \eqref{eq:TV-xi}.

Let $R$ be the function that assigns to each
point $x\in\RR^n$ the radius of the level set of $u$ 
at height $u(x)$, 
\begin{equation*}
R(x)=\left(\frac{1}{\omega_n} F(u(x))\right)^{\frac{1}{n}}.
\end{equation*}
Since $u^\star$ is a radial function, we can write
$u=u^\star\circ T$, where $T(x)=R(x) \cdot x/|x|$,
as in~\cite{FV1,FV2}.
The next two lemmas provide a bound on $|R(x)-R(y)|$.

\begin{lemma} \label{lem:grad-Phi}
Under the assumptions of Theorem~\ref{main},
if the support of $u$ has finite measure then
$R$ is of bounded variation.
The absolutely continuous part of its variation 
has density $\nabla R$, where
\begin{equation}\label{eq:grad-Phi}
|\nabla R|(x)=\begin{cases}
                               1 & \mbox{if}\ u(x)\not \in \S 
                               \>\mbox{and}\ 0<u(x)<\esup u\,,\\
0&\mbox{otherwise}\,, \end{cases}
\end{equation}
for almost every $x$. 
\end{lemma}

\begin{proof} The total variation of $R$, {given by}
$$
||DR||=\int_0^\infty \Per(\{R< t\})\, dt\,,
$$
is finite, because its value is bounded
by the radius of the support of $u$
and its (sub-) level sets $\{R<t\}$ 
are smaller balls.

The distributional derivative of $R$ is represented by
the vector-valued measure
$[DR]$. For its absolutely continuous component,
the chain rule yields on $u^{-1}((0,\esup u)\setminus \S)$
\begin{equation*}
\nabla R (x) = \frac{F'(u(x))}{n\omega_n^{1/n}F(u(x))^{1/n'}}
\, \nabla u(x) = 
-\frac{ \nabla u(x)}{|\nabla u^\star |_{\rfloor \{u^\star=u(x)\}}}\,,
\end{equation*}
see \cite[Eqs.(3.8)-(3.10)]{FV2}.
Here, $F'(t)$ is the classical derivative of $F$, and we have used 
\eqref{eq:F-prime} in the second step.
Since $u$ is an extremal, we see from \eqref{eq:grad} 
that the denominator agrees with
$|\nabla u(x)|$, and therefore $|\nabla R|=1$
almost everywhere on $u^{-1}\bigl((0,\esup u)\setminus\S\bigr)$.
Since $\lambda_1(\S)=0$, 
the gradient vanishes almost everywhere on $u^{-1}(\S)$.
\end{proof}

\smallskip 

\begin{lemma}\label{lem:lip} Under the 
assumptions of Theorem~\ref{main},
\begin{equation}\label{eq:lip}
|R(x)-R(y)|\le |x-y|+ 
\left(\frac{1}{\omega_n}
\bigl(\mu^s((u(x), u(y)]\right)^{\frac1n}
\end{equation} for almost every $x,y$
with $u(x)<u(y)$.
\end{lemma}

\begin{proof} 
By Lemma~\ref{lem:approx},
it suffices to consider functions $u$ 
whose support has finite measure and whose
distribution function has no 
singular continuous component.  Let
$$
r(\theta )=R(\theta x+(1-\theta)y)\,,\qquad 0\le \theta\le 1
$$ 
be the restriction
of $R$ to the line segment that
joins $y$ with $x$.  Since
we choose for $R$ its precise representative 
in $\BV$, the restriction is of bounded variation
and the chain rule holds for almost every choice of 
$x,y$ and almost every $\theta$~\cite[Theorem 3.107]{AFP}.
By Lemma~\ref{lem:grad-Phi}, we have 
$$
|r'(\theta)| = |\langle \nabla R( \theta x+(1-\theta)y), x-y \rangle|
\le 1\,,
$$
and obtain for the absolutely continuous part
$$
r^{ac}(1)-r^{ac}(0)\le |x-y|\,.
$$

\begin{figure}[htbp]
\centering%
\includegraphics{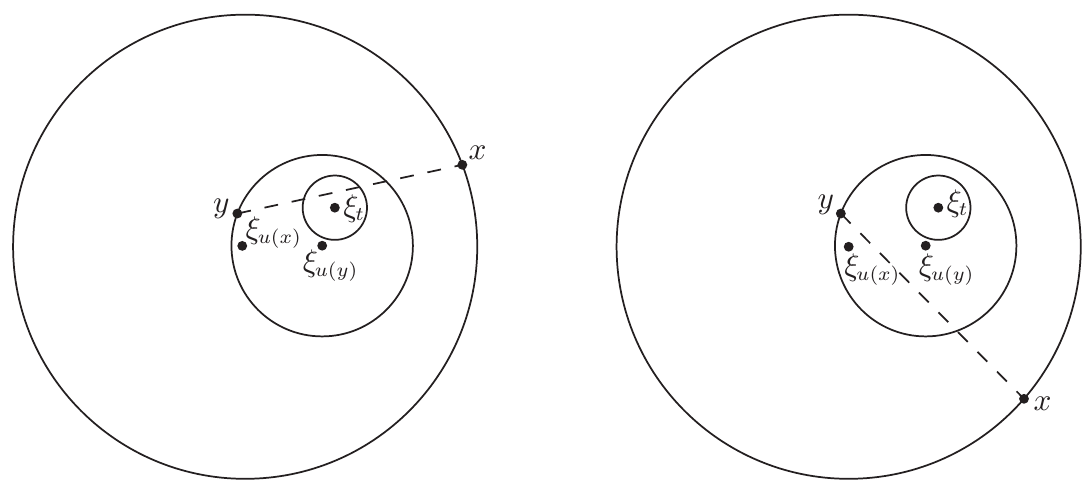}
\caption{\small The line segment from
$y$ to $x$
crosses the boundary of a higher level set either twice 
({\em left}) or never ({\em right)}.}
\label{fig:line}
\end{figure}

For the singular part, recall that $u(x)<u(y)$,
and thus $R(x)>R(y)$.
The line segment enters each level set $\{u>t\}$ 
with $t\in (u(x),u(y)]$ exactly once;
the boundary of a level set outside this range is
crossed either twice, in opposite directions, 
or not at all, see Fig.~\ref{fig:line}. 
When the line segment
enters $\{u>t\}$
for some $t\in \S$, then $R$ experiences
a positive jump of size  
$\omega_n^{-1/n}\bigl(F(t_-)^{1/n} -F(t)^{1/n}\bigr)$;
the jump is reversed upon exit. Since $\S$ is countable,
this yields
\begin{align*}
r^s(1)-r^s(0)
&= \omega_n^{-\frac1n}\sum_{t\in \S\cap (u(x),u(y)]}
\bigl(F(t_-)^{\frac1n}- F(t)^{\frac{1}{n}}\bigr)\\
&= \omega_n^{-\frac1n}
\sum_{t\in\S\cap(u(x),u(y)]}\int_{F(t_-)}^{F(t)}
\frac{1}{n} s^{\frac{1}{n}-1}\, ds\\
&\le 
\biggl(\frac{1}{\omega_n}\sum_{t\in \S\cap (u(x),u(y)]}
\bigl(F(t_-)-F(t)\bigr)\biggr)^{\frac1n}\,.
\end{align*}
We have used that the intervals $(F(t_-), F(t))$ are disjoint
and that the function $s\mapsto s^{1-1/n}$
is decreasing to move the domain
of integration to the origin. By definition, the last sum equals
$\mu^s\bigl((u(x),u(y)]\bigr)$.
The claim follows by adding the inequalities
for $r^{ac}$ and $r^s$.
\end{proof}

\smallskip 

\begin{proof} [Proof of \eqref{eq:var-xi}-\eqref{eq:TV-xi}]\
Insert Lemma~\ref{lem:lip} into \eqref{eq:xi-Phi}
to obtain the bound on $|\xi_s-\xi_t|$. The bound
on the total variation follows by maximizing over
$s,t$ and using Lemma~\ref{lem:C-S}.
\end{proof}

\smallskip 

We have used 
Lemmas~\ref{lem:grad-Phi} and \ref{lem:lip}
to show that the total variation of $\xi$ is bounded by
$(||D^sF||/\omega_n)^{1/n}$.
The proof of Lemma~\ref{lem:lip} yields
the somewhat stronger statement that 
\begin{equation}
||D\xi||\le 
\Bigl\|D^s\left(\frac{F}{\omega_n}\right)^{1/n}\Bigr\|\,.
\end{equation}
A similar computation as in Lemma~\ref{lem:lip}
shows that the singular part of the variation of $R$ is
bounded by the measure of the set
of critical points, 
\begin{equation}
||D^sR||\le  \mu^s((0,\esup u)) + \inf_{t<\esup u} F(t)\,.
\end{equation}
Here, the last term on the right
represents the possible plateau at the top, which does
not contribute to the variation of $\xi$.

%%%%%%%%%%%%%%%%%%%%%%%%%%%%%%%%%%%%%%%%%
\section{Proof of the main results}
\label{sec:main}

\begin{proof}[Proof of Theorem \ref{main}]
Since $u$ is an extremal, each level set $\{u>t\}$ is a ball
centered at $\xi_t$. Let $\xi_\infty$ be the center of the ball 
$$
\bigcap_{t\in (0,\esup u)}\,\{u>t\}
$$
(which may consist of a single point),
and consider the translation $\tau(x)=x-\xi_\infty$. 
By \eqref{eq:xi-Phi} and Lemma~\ref{lem:lip},
the distance between the level sets $\{u>t\}$
and $\{u^\star\circ\tau>t\}$ is bounded by
$$
|\xi_t-\xi_\infty|\le \left(\frac{F^s(t)}{\omega_n}\right)^{\frac1n}\,.
$$
Since the symmetric difference between
two balls of equal radius in $\RR^n$ satisfies
\begin{equation}
\label{eq:symmdiff}
\Ln((\xi+B) \bigtriangleup (\eta +B) )\le 
2\omega_{n-1}\left(\frac{\lambda_n(B)}{\omega_n}\right)^{\frac{1}{n'}}
\cdot |\xi-\eta| \,,
\end{equation}
it follows that
$$
\lambda_n(\{u>t\}\bigtriangleup \{u^\star\circ\tau>t\})
\le K_n \,F(t)^{\frac{1}{n'}}\,F^s(t)^{\frac1n}\,,
$$
where {$K_n=2\omega_{n-1}/\omega_n$}.
From Lemma~\ref{lem:Psi-2} with $\Psi(t)=t^q$, we deduce
\begin{align}\label{key<}
\notag
||u-u^\star \circ \tau||_q^q 
&\le K_n\int_0^\infty F (t)^{\frac{1}{n'}}\,
F^s(t) ^{\frac1n}\, qt^{q-1}dt\\
&\le K_n \left (\int_0^\infty F(t) \,qt^{q-1}dt\right )^{\frac{1}{n'}}
\left (\int_0^\infty F^s(t)\,qt^{q-1}\, dt\right )^{\frac1n}\\
\notag &= K_n\, ||u||_q^{q/n'} \, \|u\Chi_{\C}\|_q^{q/n}\,.  
\end{align}
We have applied 
H\"older's inequality with exponents $1/n'$ and $1/n$,
and used Lemma~\ref{lem:C-S} to interpret 
$F^s$ as the distribution function of $u\Chi_\C$.
\end{proof}

\smallskip 
The basic estimate in
the first line of \eqref{key<} can be used to derive 
other bounds on  $u-u^\star\circ\tau$, for example
\begin{equation}
\frac{||u-u^\star \circ \tau||^q_q}{\| u\|^q_q}\le 
K_n\,
\sup_{0<t<\esup u}\left (\frac{\Ln (\C\cap \{u>t\})}{\Ln 
(\{u >t\})}\right )^{\frac 1 {n}}\,.
\end{equation}
The ratio on the right hand side can be viewed
as the density of $C$ in the level set $\{u>t\}$.
It is always strictly less than one, because $\C$ does not contain
the possible plateau at $\esup u$.  
Alternately, we can interpret $n\omega_n^{1/n}F(t)^{1/n'}$ as the
perimeter of the level set $\{u>t\}$ and apply the coarea formula
to \eqref{key<} to obtain
$$
||u-u^\star \circ \tau||_q^q  \le K_n \int F^s(u)^{\frac 1 n} |\nabla u^q|\, dx\,;
$$
if $\C$ has finite measure, this implies \eqref{eq:finite}
with $p=1$.

\begin{proof}[Proof of Theorem~\ref{finite}] 
Let $\xi_\infty$ and $\tau$ be as
in the proof of Theorem~\ref{main}, and consider
Lemma~\ref{lem:Psi-1} with $\Psi(t)=t^p$.
Since the intersection between any pair of balls 
decreases with the distance of their centers, 
it follows that
$$
||u-u^\star\circ\tau||_p^p 
\le ||u^\star -u^\star\circ \tilde \tau||_p^p
\le ||\nabla u^\star||_p\cdot ||D\xi|| 
\,,
$$
where $\tilde \tau(x)=x- ||D\xi|| \, w$ for some unit 
vector $w$. The bound on the total variation of $\xi$
in \eqref{eq:TV-xi} yields the claim.
\end{proof}

\smallskip

\begin{proof}[Proof of Theorem~\ref{Morrey}] 
Let $\xi_\infty$ and $\tau$ be
as in the proof of Theorem~\ref{main}. 
Since $u(x)=u^\star(x-\xi_{u(x)})$ and 
$(u^\star\circ\tau)(x)=u^\star(x-\xi_\infty)$, 
Morrey's inequality says that
$$
|u(x)-(u^\star\circ \tau)(x)| 
\le M_{n,p} \, ||\nabla u||_p \cdot
|\xi_{u(x)}-\xi_\infty|^{1-\frac{n}{p}}\,.
$$
The claim follows with \eqref{eq:TV-xi}.
\end{proof}

%{\bf Best guess:} (based on Cianchi/Talenti)
%$M_{n,p}= n^{-1/p}\bigl(\frac{p-1}{p-n}\bigr)^{1/p'}$.

%%%%%%%%%%%%%%%%%%%%%%%%%%%%%%%%%%%%%%%%%
\section{Dirichlet-type functionals on $\BV$}
\label{sec:Dirichlet}

At last, we turn to more general convex gradient functionals.
A {\em Young function} is a nonnegative, nondecreasing
convex function $\Phi$ on $\RR_+$ with $\Phi(0)=0$.
The {\em Dirichlet functional} associated with this
Young function is defined by
$$
\F(u) = \int \Phi(|\nabla u|)\, dx\,,
$$ 
provided that the distributional gradient of $u$ is 
locally integrable.
If $\Phi$ grows linearly at infinity, the functional
is extended to $\BV$ by
$$
\F(u) = \int \Phi(|\nabla u|)\, dx + \phi \, ||D^s u||\,,
$$ 
where $\phi =\lim_{t \to\infty} \Phi(t)/t$,
and $||D^s u||$ is the singular part of the total
variation.  Then
$\F(u)$ is always well-defined but it may 
take the value $+\infty$. 
In this setting, the
P\'olya-Szeg\H{o} inequality says that
\begin{equation}
\label{eq:PS-D}
\F(u^\star)\le \F(u)\,
\end{equation}
for all $u\in\BV$~\cite{CiFu1}. 

\begin{corollary} \label{cor:main} Let $\F$ be a
Dirichlet functional on $\RR^n$ given by a strictly increasing 
Young function $\Phi$.  Let $u\in\BV$ 
be a nonnegative function that vanishes 
at infinity, and let $u^\star$ be its symmetric decreasing
rearrangement.  If
$$
\F(u)=\F(u^\star)<\infty\,,
$$
then there exists a translation $\tau$ such that
$$
\int \Psi(|u-u^\star\circ \tau|)\, dx
\le K_n\,\left(\int \Psi(u)\, dx\right)^{1/n'}
\left(\int_{\C_\Phi} \Psi(u)\, dx\right)^{1/n}\,
$$
for every Young function $\Psi$ such that $\Psi\circ u$ is integrable.
Here,
$$
\C_\Phi =\left 
\{x\in\RR^n: 0< u(x)<\esup u, |\nabla u(x)|\in \{0\}\cup V\right\}\,,
$$
and $V$ is the maximal open subset of $\RR_+$ such that 
$\Phi$ is affine on each connected component of $V$.
The constant is given by $K_n=2\omega_{n-1}/ \omega_n$.
\end{corollary}

Note that the conclusion depends on the Young function 
$\Phi$ only through the set $V$; in particular, 
{
if $\Phi$ is {\em strictly} convex then
$V=\emptyset$ and $\C_\Phi=\C$.}

\begin{proof} [Proof of Corollary~\ref{cor:main}]
Cianchi and Fusco established in~\cite{CiFu2} that 
\eqref{eq:balls} holds under the given assumptions,
i.e., the level sets of extremals  are balls. 
Moreover, \eqref{eq:grad} holds for a.e. $t\in(0,\esup u)$
such that $|\nabla u^\star|_{\rfloor \partial \{u^\star>t\}}\not \in V$,
i.e., $|\nabla u|$ is equidistributed on $\partial\{u>t\}$
for $\HH^{n-1}$-almost every $x\in\partial \{u>t\}$.  
Let
$$
\S_\Phi= \S\cup \{t>0: F'(t)=0\} \cup \{t>0: 
|\nabla u^\star|_{\rfloor \partial \{u^\star>t\}}\in V\}\,,
$$
and set $f(t)=\lambda_1((0,t)\setminus\S_\Phi)$. Then
$f\circ u^\star$ is absolutely continuous, 
its distribution function has no singular continuous
component,
and both $f\circ u^\star$ and $u^\star-f\circ u^\star$ 
 are radially decreasing functions
that vanish at infinity. By definition, 
$\C_\Phi= u^{-1}(\S_\Phi)$. Since
$$
\nabla (f\circ u) = \Chi_{\C_\Phi} \nabla u\,,
$$
the set of critical points of $f\circ u$ is given by $\C_\Phi$.
Furthermore,
$$
\F(u)=\F(f\circ u) + \F(u\!-\!f\circ u)\,,
$$
and correspondingly for $u^\star$. 
The P\'olya-Szeg\H{o} inequality holds for each
summand, and therefore
$f\circ u$ and $u-f\circ u$ must be extremals.
In particular, $f\circ u$ satisfies \eqref{eq:var-xi}. 
Since every level set $\{u>t\}$
with $t\not\in \S_\Phi$ is also a level set
of $f\circ u$, it follows that
\begin{equation}
\label{eq:var-xi-D}
|\xi_t-\xi_\infty |\le 
\left(\frac{F^s_\Phi(t)}{\omega_n} \right)^{\frac1n}\,,
\end{equation}
where 
$$
F^s_\Phi(t) = \lambda_n(\{x\in\C_\Phi: u(x)>t\})\,.
$$
By Lemma~\ref{lem:Psi-2} and H\"older's inequality,
\begin{align*}
\int \Psi(|u-u^\star\circ\tau|)\, dx &\le
K_n \int_0^\infty F(t)^{\frac{1}{n'}} \,
F^s_\Phi(t)^{\frac1n}\, 
\Psi'(t)\, dt\\
&\hskip -1.5cm 
\le K_n \!\left(\int_0^\infty F(t) \Psi'(t)\, dt\right)^{\frac{1}{n'}}\!
\left(\int_0^\infty F^s_\Phi(t)\,\Psi'(t)\, dt\right)^{\frac1n}.
\end{align*}
Using the layer-cake principle, we recognize the
integrals {in the last line} as 
$\int \Psi(u)\, dx$ and $\int_{C_\Phi} \Psi(u)\, dx$.
As in Theorems~\ref{main}-\ref{Morrey}, we 
can equivalently replace $u$ with $u^\star$ in 
these integrals and in the  definition of $\C_\Phi$.
\end{proof}

\smallskip 
If $u$ is supported on a set of
finite measure, then
we can use \eqref{eq:Psi-2} with $\Psi(t)=t$
and apply Jensen's inequality once more to conclude that
\begin{align*}
||u-u^\star\circ\tau||_1 &
\le ||\nabla u||_1 \cdot \left(\frac{\lambda_n(\C_\Phi)}
{\omega_n}\right)^{\frac1n}\\
&\le 
\Phi^{-1} \left(\frac{\F(u)}{\lambda_n(\supp u)}\right)
\cdot \lambda_n(\supp u)
\left(\frac{\lambda_n(\C_\Phi)}{\omega_n} \right)^{\frac1n}\,.
\end{align*}
Since $\C_\Phi\subset\supp u$, this
implies~\cite[Theorem 1.1]{CiFu2}. 
For $\Phi(t)=t^p$, we recover \eqref{eq:CF} with $L_n=\omega_n^{-1/n}$.

\begin{corollary} 
\label{cor:finite}
Under the assumptions of
Corollary~\ref{cor:main}, if $\C_\Phi$ has finite measure, 
then there exists a translation $\tau$ such that
$$
\int \Phi\left(|u-u^\star \circ \tau|\cdot \left(\frac{\lambda_n(\C_\Phi)}
{\omega_n}\right)^{-\frac1n}\right)\, dx \le \F(u)\,.
$$
\end{corollary}

\begin{proof} Let $\xi_\infty$ and 
$\tau$ be as in the proof of Theorem~\ref{main}.
We will show that
\begin{equation}
\label{eq:Psi-2}
\int\Psi(|u-u^\star\circ\tau|)
\, dx
\le  \int \Psi\left(|\nabla u|\cdot ||D\xi||\right)\, dx
\end{equation}
for every Young function $\Psi$ such
that the right hand side is finite. 
We then set $\Psi(t)=\Phi(t /||D\xi||)$, and use that
$$
||D\xi|| \le \left(\frac{\lambda_n(C_\Phi)}{\omega_n}\right)^{\frac1n}
$$
by \eqref{eq:var-xi-D}.

For \eqref{eq:Psi-2}, we combine \eqref{eq:L1-dist} with Lemma~\ref{lem:Psi-1}
and argue as in the proof of Theorem~\ref{finite}
that the integral on the left hand side increases if
$u$ is replaced by $u^\star\circ\tilde\tau$, 
where $\tilde \tau(x)=x-||D\xi||\,w$ for
some unit vector $w$. Since
$$
u^\star(x)-u^\star\circ\tilde\tau(x)
= \int_0^1 
\langle
\nabla u^\star(x+\theta \, ||D\xi||\,w), ||D\xi||\, w\rangle\, d\theta\,,
$$
Jensen's inequality implies that
\begin{align*}
\int \Psi(|u^\star-u^\star\circ\tilde \tau|)\, dx
&\le \int \Psi\left(
\int_0^1\bigl|\nabla u^\star(x+\theta \, ||D\xi||\, w)\bigr|\, d\theta
\cdot ||D\xi||
\!\right) dx\\
& \le \int \Psi\left(|\nabla u^\star|\cdot||D\xi||\right)\, dx \\
&\le  \int \Psi\left(|\nabla u|\cdot ||D\xi||\right)\, dx\,.
\end{align*}
The last step holds by the P\'olya-Szeg\H{o} inequality in
\eqref{eq:PS-D}.  
\end{proof}

%\bibliographystyle{plain}
%\bibliography{bibBF.bib}

\end{document}